\documentclass[11pt]{article}
\usepackage{amssymb}
\usepackage{amsthm}
\usepackage{amsmath}
\usepackage{fullpage,hyperref}
\usepackage{graphicx,url,amsmath,amsthm,amsfonts,amssymb,subfigure,bbm}

\newcommand\remove[1]{}

\newcommand{\rnote}[1]{}

\newcommand{\E}{\mathbb E}
\newcommand{\Aut}{\mathsf{Aut}}

\newcommand{\dist}{\mathrm{dist}}

\newcommand{\R}{\mathbb{R}}
\newcommand{\Lip}{\mathrm{Lip}}

\newcommand{\len}{\mathsf{len}}

\newcommand{\tri}{\mathrm{tri}}

\DeclareMathOperator{\diam}{diam}

\newtheorem{theorem}{Theorem}[section]
\newtheorem{lemma}[theorem]{Lemma}

\newtheorem{claim}[theorem]{Claim}

\newtheorem{corollary}[theorem]{Corollary}

\newtheorem{definition}[theorem]{Definition}
\newtheorem{observation}[theorem]{Observation}

\date{}

\begin{document}

\title{{\bf On the optimality of gluing over scales}}
\author{Alexander Jaffe\thanks{Research partially supported by NSF CCF-0644037 and a Sloan Research Fellowship.}
 \and James R. Lee$^*$ \and Mohammad Moharrami$^{*}$ }
\date{}

\maketitle

\begin{abstract}
We show that for every $\alpha > 0$, there exist $n$-point metric spaces
$(X,d)$ where every ``scale'' admits a Euclidean embedding with distortion at most
$\alpha$, but the whole space requires distortion at least $\Omega(\sqrt{\alpha \log n})$.
This shows that the scale-gluing lemma [Lee, SODA 2005]
is tight, and disproves a conjecture stated there.
This matching upper bound was known to be tight at both endpoints, i.e. when $\alpha = \Theta(1)$ and
$\alpha = \Theta(\log n)$, but nowhere in between.

More specifically,
we exhibit $n$-point spaces with doubling constant $\lambda$
requiring Euclidean distortion $\Omega(\sqrt{\log \lambda \log n})$, which also
shows that the
technique of ``measured descent'' [Krauthgamer, et. al., {\em Geometric and Functional Analysis}]
is optimal.  We  extend this to
$L_p$ spaces with $p > 1$, where one requires distortion at least
$\Omega((\log n)^{1/q} (\log \lambda)^{1-1/q})$ when $q=\max\{p,2\}$,
a result which is tight for every $p > 1$.
\end{abstract}
\section{Introduction}

Suppose one is given a collection of mappings from some finite metric space $(X,d)$
into a Euclidean space, each of which reflects the geometry at some ``scale'' of $X$.
Is there a non-trivial way of gluing these mappings together to form
a global mapping which reflects the entire geometry of $X$?   The answers
to such questions have played a fundamental role in the best-known approximation
algorithms for Sparsest Cut \cite{KLMN05,Lee05,CGR05,ALN05} and Graph Bandwidth \cite{Rao99,KLMN05,LeeVol},
and have found applications in approximate multi-commodity max-flow/min-cut theorems
in graphs \cite{Rao99,KLMN05}.  In the present paper, we show that the
approaches of \cite{KLMN05} and \cite{Lee05} are optimal, disproving a conjecture
stated in \cite{Lee05}.

\medskip

Let $(X,d)$ be an $n$-point metric space, and suppose that for every $k \in \mathbb Z$,
we are given a non-expansive mapping $\phi_k : X \to L_2$ which satisfies the following.
For every $x,y \in X$ with $d(x,y) \geq 2^k$, we have
$$
\|\phi_k(x)-\phi_k(y)\| \geq \frac{2^k}{\alpha}.
$$
The Gluing Lemma of \cite{Lee05} (generalizing the approach of \cite{KLMN05}) shows that the existence of such a collection $\{\phi_k\}$ yields
a Euclidean embedding of $(X,d)$ with distortion $O(\sqrt{\alpha \log n})$.  (See Section \ref{sec:prelims}
for the relevant definitions on embeddings and distortion.)
This is known to be tight when $\alpha=\Theta(1)$ \cite{NR03} and also when $\alpha=\Theta(\log n)$ \cite{LLR95,AR98},
but nowhere in between.  In fact, in \cite{Lee05}, the second named author
conjectured that one could achieve $O(\alpha + \sqrt{\log n})$ (this is indeed stronger,
since one can always construct $\{\phi_k\}$ with $\alpha = O(\log n)$).

In the present paper, we give a family of examples which shows that the $\sqrt{\alpha \log n}$ bound is tight for any
dependence $\alpha(n) = O(\log n)$.  In fact, we show more.  Let $\lambda(X)$ denote the {\em doubling
constant} of $X$, i.e. the smallest number $\lambda$ so that every open ball in $X$ can be covered
by $\lambda$ balls of half the radius.  In \cite{KLMN05}, using the method of ``measure descent,''
the authors show that $(X,d)$ admits a Euclidean embedding with distortion $O(\sqrt{\log \lambda(X) \log n})$.
(This is a special case of the Gluing Lemma since one can always find $\{\phi_k\}$
with $\alpha = O(\log \lambda(X))$ \cite{GKL03}).  Again, this bound was known to be tight
for $\lambda(X) = \Theta(1)$ \cite{Laakso,Lang,GKL03} and $\lambda(X) = n^{\Theta(1)}$ \cite{LLR95,AR98},
but nowhere in between.  We provide the matching lower bound for any dependence of $\lambda(X)$ on $n$.
We also generalize our method to give tight lower bounds on $L_p$ distortion
for every fixed $p > 1$.

\medskip
\noindent
{\bf Construction and analysis.}  In some sense, our lower bound examples are an
interpolation between the multi-scale method of \cite{NR03} and \cite{Laakso},
and the expander Poincar\'e inequalities of \cite{LLR95,AR98,MatExp}.
We start with a vertex-transitive expander graph $G$ on $m$ nodes.  If $D$ is the diameter
of $G$, then we create $D+1$ copies $G^{1}, G^{2}, \ldots, G^{D+1}$ of $G$
where $u \in G^{i}$ is connected to $v \in G^{i+1}$ if $(u,v)$ is an edge in $G$,
or if $u=v$.  We then connect a vertex $s$ to every node in $G^{1}$ and
a vertex $t$ to every node in $G^{D+1}$ by edges of length $D$.
This yields the graph $\vec G$ described in Section \ref{sec:vecG}.

In Section \ref{sec:LB}, we show that whenever there is a non-contracting
embedding $f$ of $\vec G$ into $L_2$, the following holds.
If $\gamma = \frac{\|f(s)-f(t)\|}{d_{\vec G}(s,t)}$, then some edge
of $\vec G$ gets stretched by at least $\sqrt{\gamma^2 + \Omega(\log m)^2}$,
i.e. there is a ``stretch increase.''
This is proved by combining the uniform convexity of $L_2$ (i.e. the Pythagorean
theorem), with the well-known contraction property of expander graphs
mapped into Hilbert space.  To convert the ``average'' nature of this contraction
to information about a specific edge, we symmetrize the embedding
over all automorphisms of $G$ (which was chosen to be vertex-transitive).

 To exploit this stretch increase
recursively, we construct a graph $\vec G^{\oslash k}$ inductively as follows:
$\vec G^{\oslash k}$ is formed by replacing every edge of $\vec G^{\oslash k-1}$
by a copy of $\vec G$ (see Section \ref{sec:oslash} for the formal definitions).
Now a simple induction shows that in a non-contracting embedding
of $\vec G^{\oslash k}$, there must be an edge stretched by at least $\Omega(\sqrt{k} \log m)$.
In Section \ref{sec:Lp}, a similar argument is made for $L_p$ distortion, for $p > 1$,
but here we have to argue about ``quadrilaterals'' instead of ``triangles''
(in order to apply the uniform convexity inequality in $L_p$), and it requires
slightly more effort to find a good quadrilateral.

Finally, we observe that if $\widetilde{G}$ is the graph formed by adding two tails of length $3D$ hanging off $s$ and $t$
in $\vec G$, then (following the analysis of \cite{Laakso,Lang}), one has
$\log \lambda(\widetilde{G}^{\oslash k}) \lesssim  \log m$.
The same lower bound analysis also works for $\widetilde{G}^{\oslash k}$, so
since
$n = |V(\widetilde{G}^{\oslash k})| = 2^{\Theta(k \log m)}$, the lower bound is
$$\sqrt{k} \log m \approx \sqrt{\log m \log n} \gtrsim \sqrt{\log \lambda(\widetilde{G}^{\oslash k}) \log n},$$
completing the proof.

\subsection{Preliminaries}
\label{sec:prelims}

   For a graph $G$, we will use $V(G), E(G)$ to denote the sets of vertices and edges of $G$, respectively.
    Sometimes we will equip $G$ with a non-negative length function $\len : E(G) \to \mathbb R_+$,
    and we let $d_{\len}$ denote the shortest-path (semi-)metric on $G$.
    We refer to the pair $(G,\len)$ as a {\em metric graph,} and often $\len$ will be implicit, in which
case we use $d_G$ to denote the path metric.
We use $\Aut(G)$ to denote the group of automorphisms of $G$.

Given two expressions $E$ and $E'$ (possibly depending on a number of parameters), we write $E = O(E')$ to mean that $E \leq C E'$
for some constant $C > 0$ which is independent of the parameters. Similarly, $E = \Omega(E')$ implies that $E \geq C E'$ for some $C > 0$.
We also write $E \lesssim E'$ as a synonym for $E = O(E')$.  Finally, we write $E \approx E'$ to denote
the conjunction of $E \lesssim E'$ and $E \gtrsim E'$.

\medskip
\noindent
{\bf Embeddings and distortion.}
If $(X,d_X),(Y,d_Y)$ are metric spaces, and
$f : X \to Y$, then we write $$\|f\|_\Lip = \sup_{x \neq y \in X} \frac{d_Y(f(x),f(y))}{d_X(x,y)}.$$
If $f$ is injective, then the {\em distortion of $f$} is
defined by $\dist(f) = \|f\|_\Lip \cdot \|f^{-1}\|_\Lip$.
A map with distortion $D$ will sometimes be referred to as {\em $D$-bi-lipschitz.}
If $d_Y(f(x),f(y)) \leq d_X(x,y)$ for every $x,y \in X$,
we say that $f$ is {\em non-expansive.}
If $d_Y(f(x),f(y)) \geq d_X(x,y)$ for every $x,y \in X$,
we say that $f$ is {\em non-contracting.}
For a metric space $X$, we use $c_p(X)$ to denote the least distortion required to embed $X$ into some $L_p$ space.

Finally, for $x \in X$, $r \in \R_+$, we define the open ball $B(x,r) = \{y \in X : d(x,y) < r\}$. Recall that the {\em doubling constant} of a metric space $(X,d)$ is the
infimum over all values $\lambda$ such that every ball in $X$ can be covered
by $\lambda$ balls of half the radius.  We use $\lambda(X,d)$ to denote this value.

We now state the main theorem of the paper.

\begin{theorem}
For any positive nondecreasing function $\lambda(n)$, there exists a family of $n$-vertex metric graphs
$\widetilde{G}^{\oslash k}$ such that $\lambda(\widetilde{G}^{\oslash k}) \lesssim \lambda(n)$,
and for every fixed $p > 1$, $$c_p(\widetilde{G}^{\oslash k}) \gtrsim (\log n)^{1/q} (\log \lambda(n))^{1-1/q},$$ where $q = \max\{p,2\}$.
\end{theorem}

\section{Metric construction}

\subsection{$\oslash$-products}
\label{sec:oslash}

An $s$-$t$ graph $G$ is
a graph which has two distinguished vertices $s,t \in V(G)$.  For an $s$-$t$ graph,
we use $s(G)$ and $t(G)$ to denote the vertices labeled $s$ and $t$, respectively.
We define the length of an $s$-$t$ graph $G$ as $\len(G) = d_{\len}(s,t)$.

\begin{figure}
\centering
\includegraphics[width=12cm]{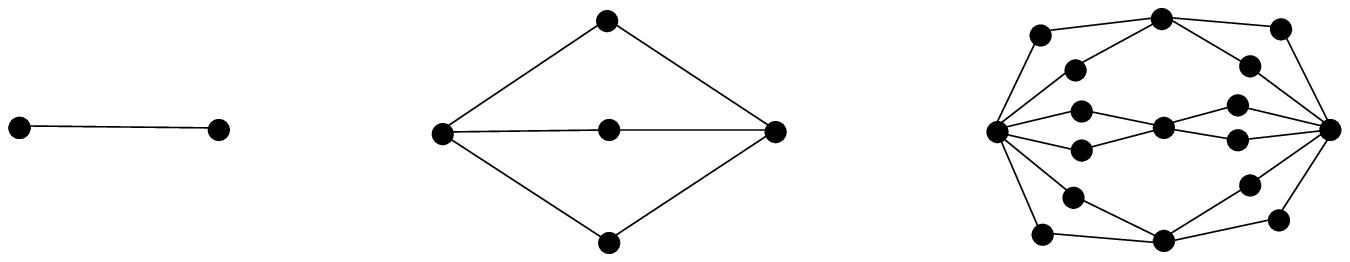}
\caption{A single edge $H$, $H \oslash K_{2,3}$, and $H \oslash K_{2,3} \oslash K_{2,2}$.}
\label{fig:oslash}
\end{figure}

\begin{definition}[Composition of $s$-$t$ graphs]
Given two $s$-$t$ graphs $H$ and $G$, define $H \oslash G$ to be the $s$-$t$
graph obtained by replacing each edge $(u,v) \in E(H)$ by a copy of $G$
(see Figure \ref{fig:oslash}).
Formally,
\begin{itemize}
\item $V(H \oslash G) = V(H) \cup \left( \vphantom{\bigoplus} E(H) \times (V(G) \setminus \{s(G),t(G)\}) \right).$
\item For every edge $e=(u,v) \in E(H)$, there are $|E(G)|$ edges,
\begin{eqnarray*}
&&
\!\!\!\!\!\!\!\!\!\!\!\!\!\!\!
\left\{ \left(\vphantom{\bigoplus} (e,v_1), (e,v_2) \right)\,|\, (v_1,v_2) \in E(G) \textrm{ and } v_1, v_2 \notin \{s(G),t(G)\}\right\}
\cup
\\
&&  \left\{ \left(\vphantom{\bigoplus} u, (e, w)\right) \,|\, (s(G),w) \in E(G) \right\}
\cup \left\{ \left(\vphantom{\bigoplus}  (e, w),v\right) \,|\, (w,t(G)) \in E(G) \right\}
\end{eqnarray*}
\item $s(H \oslash G) = s(H)$ and $t(H \oslash G) = t(H)$.
\end{itemize}

If $H$ and $G$ are equipped with length functions $\len_H, \len_G$, respectively,
we define $\len=\len_{H \oslash G}$ as follows.  Using the preceding notation,
for every edge $e = (u,v) \in E(H)$,
\begin{eqnarray*}
\len\left((e,v_1),(e,v_2)\right) &=& \frac{\len_H(e)}{d_{\len_G}(s(G),t(G))} \len_G(v_1,v_2) \\
\len\left(u,(e,w)\right) &=& \frac{\len_H(e)}{d_{\len_G}(s(G),t(G))} \len_G(s(G),w) \\
\len\left((e,w),v\right) &=& \frac{\len_H(e)}{d_{\len_G}(s(G),t(G))} \len_G(w,t(G)).
\end{eqnarray*}
This choice implies that $H \oslash G$ contains an isometric copy of $(V(H),d_{\len_H})$.
\end{definition}

Observe that there is some ambiguity in the definition above, as there are two ways
to substitute an edge of $H$ with a copy of $G$, thus we assume
that there exists some arbitrary orientation of the edges of $H$.  However, for our purposes
the graph $G$ will be symmetric, and thus the orientations are irrelevant.

\begin{definition}[Recursive composition]
For an $s$-$t$ graph $G$ and a number $k \in \mathbb N$, we define
$G^{\oslash k}$ inductively by letting $G^{\oslash 0}$ be a single edge of unit length,
and setting $G^{\oslash k} = G^{\oslash k-1} \oslash G$.
\end{definition}

The following result is straightforward.

\begin{lemma}[Associativity of $\oslash$]
For any three graphs $A,B,C$, we have $(A \oslash B) \oslash C = A
\oslash (B \oslash C)$, both graph-theoretically and as metric
spaces.
\end{lemma}

\begin{definition}
   For two graphs $G$, $H$, a subset of vertices $X \subseteq V(H)$ is said to be a {\it copy} of $G$ if there exists a
   bijection $f : V(G) \rightarrow X$ with distortion 1, i.e. $d_{H}(f(u),f(v)) = C\cdot d_{G}(u,v)$ for some constant $C > 0$.
\end{definition}

   Now we make the following two simple observations about copies of $H$ and $G$ in $H \oslash G$.
   \begin{observation}\label{obs1}
   The graph $H \oslash G$ contains $|E(H)|$ distinguished copies of the graph $G$, one copy corresponding to each edge in $H$.
   \end{observation}

   \begin{observation}\label{obs2}
   The subset of vertices $V(H) \subseteq V(H \oslash G)$ form an isometric {\it copy} of $H$.
   \end{observation}

\subsection{A stretched $\vec G$}
\label{sec:vecG}

% should give a Laakso-ish construction here,
% and then prove that the doubling constant only depends on $G$.

Let $G=(V,E)$ be an unweighted graph, and put $D = \diam(G)$.
We define a metric $s$-$t$ graph $\vec G$ which has $D+1$
layers isomorphic to $G$, with edges between the layers,
and a pair of endpoints $s,t$.  Formally,
\begin{eqnarray*}
V(\vec G) &=& \{ s,t \} \cup \{ v^{(i)}: v \in V, i \in [D+1]\} \\
E(\vec G) &=& \{ (s, v^{(1)}), (v^{(D+1)}, t) : v \in V \} \\
&&\qquad\cup
\left\{ \vphantom{\bigoplus} (u^{(i)}, v^{(i+1)}), (u^{(j)}, v^{(j)}) :  (u,v) \in E, i \in [D],
j \in [D+1] \right\} \\
&& \qquad\cup\,\, \{ (v^{(i)}, v^{(i+1)}) : v \in V, i \in [D] \}.
\end{eqnarray*}
We put $\len(s,v^{(1)})\!=\!\len(v^{(D+1)},t)\!=\!D$ for $v\! \in\! V$,
$\len(u^{(i)},v^{(i+1)})\! =\! \len(u^{(j)},v^{(j)})\! = 1$ for $(u,v) \in E$, $i \in [D]$, $j \in [D+1]$ and
$\len(v^{(i)},v^{(i+1)})=1$ for $v \in V, i \in [D]$.
We refer to edges of the form $(u^{(i)}, v^{(i)})$ as {\em vertical edges.}
All other edges are called {\em horizontal edges.}
In particular, there are $D+1$ copies $G^{(1)},\ldots,G^{(D+1)}$ of $G$ in $\vec G$ which
are isometric to $G$ itself, and their edges are all vertical.

\medskip
\noindent
{\bf A doubling version, following Laakso.}
Let $\vec G$ be a stretched graph as in Section \ref{sec:vecG}, with $D=\diam(G)$,
and let $s' = s(\vec G), t' = t(\vec G)$. Consider a new metric $s$-$t$ graph $\widetilde{G}$,
which has two new vertices $s,t$ and two new edges
$(s,s'),(t',t)$ with $\len(s,s')=\len(t',t)=3D$.

\begin{claim} \label{claim:doublingbound}
For any graph $G$ with $|V(G)| = m$, and any $k \in \mathbb N$, we have
$\log \lambda(\widetilde{G}^{\oslash k}) \lesssim \log m$.
\end{claim}

The proof of the claim is similar to \cite{Laakso,Lang}, and follows from the following three results.

\medskip

We define
$\tri(G) = \max_{v\in V(G)} (d_{\len}(s,v)+d_{\len}(v,t))$. For any graph $G$, we have $\len(\widetilde G)=d(s,t)=9D$, and it is not hard to verify that $\tri(\widetilde{G}^{\oslash k})\leq \len(\widetilde{G}^{\oslash k})(1+{1 \over 9D-1 })$. For convenience,
let $G_0$ be the top-level copy of $\widetilde G$ in $\widetilde{G}^{\oslash k}$, and $H$ be the graph $\widetilde{G}^{\oslash k-1}$.
Then for any $e \in E(G_0)$, we refer to the copy of $H$ along edge $e$ as $H_e$.

\begin{observation} \label{obs:larger}
If $r > {\tri(\widetilde{G}^{\oslash k})\over 3}$, then the ball $B(x,r)$ in $\widetilde{G}^{\oslash k}$ may be covered by at most $|V(\widetilde{G})|$ balls of radius $r/2$.
\end{observation}
\begin{proof}
For any $e \in E(G_0)$, we have $r > {\len(e) \over \len(H)} \tri(H)$, so every point in $H_{e}$ is less than $r/2$ from an endpoint of $e$. Thus all of $\widetilde{G}^{\oslash k}$ is covered by placing balls of radius ${\tri(\widetilde{G}^{\oslash k})\over 6}$ around each vertex of $G_0$.
\end{proof}

\begin{lemma} \label{lemma:endpoints}
If $s \in B(x,r)$, then one can cover the ball $B(x,r)$ in $\widetilde{G}^{\oslash k}$ with at most $|E(\widetilde{G})| |V(\widetilde{G})|$ balls of radius $r/2$.
\end{lemma}

\begin{proof}
First consider the case in which $r > {\len(\widetilde{G}^{\oslash k})\over 6}$. Then for any edge $e$ in $\widetilde{G}^{\oslash k}$, we have $r > {\len(e)\over \len(H)} \cdot \frac{\tri(H)}{3}$. Thus by Observation \ref{obs:larger}, we may cover $H_e$ by $|V(\widetilde{G})|$ balls of radius $r/2$. This gives a covering of all of $\widetilde{G}^{\oslash k}$ by at most $|E(\widetilde{G})| |V(\widetilde{G})|$ balls of radius $r/2$.

Otherwise, assume ${\len(\widetilde{G}^{\oslash k})\over 6} \geq r$. Since $s \in B(x,r)$, but $2r \leq \frac{\len(\widetilde{G}^{\oslash k})}{3}$, the ball must be completely contained inside $H_{(s,s')}$. By induction, we can find a sufficient cover of this smaller graph.
\end{proof}

\begin{lemma} \label{lem:maindoublinglem}
We can cover any ball $B(x,r)$ in $\widetilde{G}^{\oslash k}$ with at most $2 |V(\widetilde{G})|  |E(\widetilde{G})|^2$ balls of radius $r/2$.
\end{lemma}

\begin{proof}
We prove this lemma using induction. For $\widetilde{G}^{\oslash0}$, the claim holds trivially. Next, if any $H_{e}$ contains all of $B(x,r)$, then by induction we are done. Otherwise, for each $H_{e}$ containing $x$, $B(x,r)$ contains an endpoint of $e$. Then by Lemma \ref{lemma:endpoints}, we may cover $H_{e}$ by at most $|E(\widetilde{G})| |V(\widetilde{G})|$ balls of radius $r/2$. For all other edges $e' = (u,v)$, $x \notin H_{e'}$, so we have:
\[
V(H_{e'})\cap B(x,r)\subseteq B(v,\max(0,r-d(x,v)))\cup B(u,\max(0,r-d(x,u))).
\]
Thus, using Lemma \ref{lemma:endpoints} on both of the above balls, we may cover $V(H_{e'})\cap B(x,r)$ by at most
$2|E(\widetilde G)| |V(\widetilde G)|$ balls of radius $r/2$. Hence, in total, we need at most $2 |V(\widetilde{G})| |E(\widetilde{G})|^2$
balls of radius $r/2$ to cover all of $B(x,r)$.

\end{proof}

\begin{proof} [Proof of Claim \ref{claim:doublingbound}]
First note that $|V(\widetilde G)| = m (D + 1) + 2 \lesssim m^2$.
By Lemma \ref{lem:maindoublinglem}, we have
\[
\lambda(\widetilde{G}^{\oslash k}) \leq 2 |V(\widetilde{G})|  |E(\widetilde{G})|^2 \leq
2 |V(\widetilde{G})|^5 \lesssim m^{10}.
\]
Hence $\log \lambda(\widetilde{G}^{\oslash k}) \lesssim \log m$.
\end{proof}

\section{Lower bound}
\label{sec:LB}

For any $\pi \in \Aut(G)$, we define a corresponding automorphism $\tilde \pi$ of $\tilde G$
by $\tilde \pi(s)=s$, $\tilde \pi(t)=t$, $\tilde \pi(s')=s'$, $\tilde \pi(t')=t'$, and $\tilde \pi(v^{(i)}) = \pi(v)^{(i)}$ for $v \in V, i \in [D+1]$.

\begin{lemma}\label{lem:stretch}
Let $G$ be a vertex transitive graph.
Let $f : V(\widetilde G) \to L_2$ be an injective mapping
and define
$\bar f : V(\widetilde G) \to L_2$ by
$$
\bar f(x) = \frac{1}{\sqrt{|\Aut(G)|}}\left(f(\widetilde \pi x)\vphantom{\bigoplus}\right)_{\pi \in \Aut(G)}.
$$
Let $\beta$ be such that for every $i \in [D+1]$ there exists a vertical edge $(u^{(i)}, v^{(i)})$
with $\|\bar f(u^{(i)})-\bar f(v^{(i)})\| \geq \beta$.
Then there exists a horizontal edge $(x,y) \in E(\widetilde G)$ such that
\begin{equation}\label{eq:stretch}
\frac{\|\bar f(x)-\bar f(y)\|^2}{d_{\widetilde G}(x,y)^2} \geq \frac{\bar \|f(s)-\bar f(t)\|^2}{d_{\widetilde G}(s,t)^2} + \frac{\beta^2}{36}
\end{equation}
\end{lemma}

\begin{proof}
Let $D = \diam(G)$.
We first observe three facts about $\bar f$, which rely
on the fact that when $\Aut(G)$ is transitive, for every $x \in V$,
the orbits $\{\pi(x)\}_{\pi \in \Aut(G)}$ all have the same cardinality.

\begin{enumerate}
\item[(F1)] $\|\bar f(s)-\bar f(t)\| = \|f(s)-f(t)\|$
\item[(F2)] For all $u,v \in V$,
\begin{eqnarray*}
\|\bar f(s)-\bar f(v^{(1)})\|&=&\|\bar f(s)-\bar f(u^{(1)})\|,\\
\|\bar f(t)-\bar f(v^{(D+1)})\|&=&\|\bar f(t)-\bar f(u^{(D+1)})\|.
\end{eqnarray*}
\item[(F3)] For every $u,v \in V$, $i\in [D]$,
$$\|\bar f(v^{(i)})-\bar f(v^{(i+1)})\|=\|\bar f(u^{(i)})-\bar f(u^{(i+1)})\|.$$
\item[(F4)] For every pair of vertices $u,v \in V$ and $i \in [D+1]$,
$$
\langle \bar f(s)-\bar f(t), \bar f(u^{(i)}) - \bar f(v^{(i)}) \rangle = 0.
$$
\end{enumerate}

Let $z = {\bar f(s)-\bar f(t)\over \| \bar f(s)-\bar f(t)\|}$.  Fix some $r \in V$ and
let $\rho_0 = |\langle z, \bar f(s)-\bar f(r^{(1)})\rangle|$,
$\rho_i =  |\langle z, \bar f(r^{(i)})-\bar f(r^{(i+1)})\rangle|$
for $i=1,2,\ldots,D$ and
$\rho_{D+1} = |\langle z, \bar f(t)-\bar f(r^{(D+1)})\rangle|$.
Note that, by (F2) and (F3) above, the values $\{\rho_i\}$ do not
depend on the representative $r \in V$.
In this case, we have
\begin{equation}\label{eq:tri}
\sum_{i=0}^{D+1} \rho_i \geq \|\bar f(s)-\bar f(t)\| = 9\gamma D,
\end{equation}
where we put $\gamma = \frac{\|\bar f(s)-\bar f(t)\|}{d_{\widetilde G}(s,t)}$.  Note that $\gamma > 0$
since $f$ is injective.

Recalling that $d_{\widetilde G}(s,t) = 9D$ and $d_{\widetilde G}(s,r^{(1)})=4D$, observe that
if $\rho_0^2 \geq \left(1+\frac{\beta^2}{36\gamma^2}\right) (4 \gamma D)^2$, then
$$
\max \left(\frac{\|\bar f(s)-\bar f(s'))\|^2}{d_{\widetilde G}(s,s')^2} ,\frac{\|\bar f(s')-\bar f(r^{(1)})\|^2}{d_{\widetilde G}(s',r^{(1)})^2}\right) \geq \gamma^2 + \frac{\beta^2}{36},
$$
verifying \eqref{eq:stretch}.  The symmetric argument holds for $\rho_{D+1}$, thus
we may assume that $$\rho_0, \rho_{D+1} \leq  4\gamma D\sqrt{1+\frac{\beta^2}{36\gamma^2}}
\leq 4\gamma D \left(1+\frac{\beta^2}{72\gamma^2}\right).
$$
In this case, by \eqref{eq:tri}, there must exist an index $j \in [D]$ such that
$$\rho_j \geq \left(1-\frac{8 \beta^2}{72 \gamma^2}\right) \gamma=\left(1-\frac{\beta^2}{9 \gamma^2}\right) \gamma.$$

Now, consider a vertical edge $(u^{(j+1)}, v^{(j+1)})$ with
$\|\bar f(u^{(j)})-\bar f(v^{(j)})\| \geq \beta$, and let $$u'=\bar f(u^{(j)})+ \langle z, \bar f(u^{(j)}) - \bar f(u^{(j+1)}) \rangle\,z\,.$$
From (F4), and the Pythagorean inequality we have,
\begin{eqnarray*}
&& \!\!\!\!\!\!\!\!\!\!\!\!\!\!\!\!\!\!\!\!\!\!\!\!\!\!\!\!\!\!\!\!\!\!\!\!\!\!\!\!\!\!\!\!\!\!\!\!\!\!\!\!\!\!\!\!\!\!\!\!
\!\!\!\!\!\!\!\!\!\!\!
\max(\|\bar f(u^{(j)})\!-\bar f(u^{(j+1)})\|^2,\|\bar f(u^{(j)})-\bar f(v^{(j+1)})\|^2) = \\
\|\bar f(u^{(j)})\!-u'\|^2 \!\!\! &+& \!\!\! \max({\|u'\!-\bar f(u^{(j+1)})\|^2,\|u'-\bar f(v^{(j+1)})\|^2)}) \\
&\geq&  \rho_j^2 + \frac{\beta^2}{4} \\
&\geq& \left(1-\frac{2 \beta^2}{9\gamma^2}\right) \gamma^2 + {\beta^2\over 4} \\
&\geq& \gamma^2 + \frac{\beta^2}{36},
\end{eqnarray*}
again verifying \eqref{eq:stretch} for one of the two edges $(u^{(j)}, v^{(j+1)})$ or $(u^{(j)}, u^{(j+1)})$.
\end{proof}

The following lemma is well-known, and follows from the variational
characterization of eigenvalues (see, e.g. \cite[Ch. 15]{Mat01}).

\begin{lemma}\label{lem:poincare}
If $G=(V,E)$ is a $d$-regular graph with second Laplacian eigenvalue $\mu_2(G)$,
then for any mapping $f : V \to L_2$, we have
\begin{equation}\label{eq:poincare}
\E_{x,y \in V} \,\|f(x)-f(y)\|^2 \lesssim \frac{d}{\mu_2(G)}\, \E_{(x,y) \in E} \,\|f(x)-f(y)\|^2
\end{equation} \
\end{lemma}

The next lemma shows that when we use an expander graph, we get a significant
increase in stretch for edges of $\widetilde G$.

\begin{lemma}\label{lem:find}
Let $G=(V,E)$ be a $d$-regular vertex-transitive graph with $m=|V|$ and $\mu_2=\mu_2(G)$.
If $f : V(\widetilde G) \to L_2$ is any non-contractive mapping, then
there exists a horizontal edge $(x,y) \in E(\widetilde G)$ with
\begin{equation}
\label{eq:long}
\frac{\|f(x)-f(y)\|^2}{d_{\widetilde G}(x,y)^2} \geq \frac{\|f(s)-f(t)\|^2}{d_{\widetilde G}(s,t)^2} + \Omega\left(\frac{\mu_2}{d} (\log_d m)^2\right).
\end{equation}
\end{lemma}

\begin{proof}
We need only prove the existence of an $(x,y) \in E(\widetilde G)$ such that \eqref{eq:long} is satisfied
for $\bar f$ (as defined in Lemma \ref{lem:stretch}), as this implies it is also satisfied
for $f$ (possibly for some other edge $(x,y)$).

Consider any layer $G^{(i)}$ in $\widetilde G$, for $i \in [D+1]$.
Applying \eqref{eq:poincare} and using the fact that $f$ is non-contracting, we have
\begin{eqnarray*}
\E_{(u,v) \in E} \,\|\bar f(u^{(i)})-\bar f(v^{(i)})\|^2 &=&
\E_{(u,v) \in E} \,\|f(u^{(i)})-f(v^{(i)})\|^2 \\
&\gtrsim& \frac{\mu_2}{d} \,\E_{u,v \in V} \,\|f(u^{(i)})-f(v^{(i)})\|^2 \\
&\geq & \frac{\mu_2}{d} \,\E_{u,v \in V} \,d_{G}(u,v)^2 \\
&\gtrsim & \frac{\mu_2}{d} (\log_d m)^2.
\end{eqnarray*}
In particular, in every layer $i \in [D+1]$, at least one vertical edge
$(u^{(i)}, v^{(i)})$ has $\|\bar f(u^{(i)})-\bar f(v^{(i)})\| \gtrsim \sqrt{\frac{\mu_2}{d}} \log_d m.$
Therefore the desired result follows from Lemma \ref{lem:stretch}.
\end{proof}

We now to come our main theorem.

\begin{theorem}
If $G=(V,E)$ is a $d$-regular, $m$-vertex, vertex-transitive graph with $\mu_2=\mu_2(G)$, then
$$c_2(\widetilde G^{\oslash k}) \gtrsim \sqrt{\frac{\mu_2 k}{d}} \log_d m.$$
\end{theorem}

\begin{proof}
Let $f : V(\widetilde G^{\oslash k}) \to L_2$ be any non-contracting embedding.
The theorem follows almost immediately by induction:  Consider the top level
copy of $\widetilde G$ in $\widetilde G^{\oslash k}$, and call it $G_0$.  Let $(x,y) \in E(G_0)$ be the horizontal edge for which $\|f(x)-f(y)\|$ is longest.  Clearly this edge
spans a copy of $\widetilde G^{\oslash k-1}$, which we call $G_1$.
By induction and an application of Lemma \ref{lem:find}, there exists a (universal) constant $c > 0$ and an edge $(u,v) \in E(G_1)$ such that

\begin{eqnarray*}
\frac{\|f(u)-f(v)\|^2}{d_{\widetilde G^{\oslash k}}(u,v)^2}
	&\geq& \frac{c\mu_2 (k-1)}{d} (\log_d m)^2 + \frac{\|f(x)-f(y)\|^2}{d_{\widetilde G^{\oslash k}}(x,y)^2} \\
&\geq& \frac{c\mu_2 (k-1)}{d}(\log_d m)^2+\frac{c\mu_2}{d}
	(\log_d m)^2+ \frac{ \|f(s) - f(t)\|^2 }{ d_{\tilde{G}^{\oslash k}} (s,t) },
\end{eqnarray*}

%$$
%\frac{\|f(u)-f(v)\|^2}{d_{\widetilde G^{\oslash k}}(u,v)^2} \geq \frac{c\mu_2 (k-1)}{d} (\log_d m)^2
%+ \frac{\|f(x)-f(y)\|^2}{d_{\widetilde G^{\oslash k}}(x,y)^2}
%\geq $$ $$\frac{c\mu_2 (k-1)}{d}(\log_d m)^2+\frac{c\mu_2}{d}
%(\log_d m)^2+ \frac{ \|f(s) - f(t)\|^2 }{ d_{\tilde{G}^{\oslash k}} (s,t) },
%$$
completing the proof.
\end{proof}

\begin{corollary}\label{cor:main}
If $G=(V,E)$ is an $O(1)$-regular $m$-vertex, vertex-transitive graph with $\mu_2 = \Omega(1)$, then
$$
c_2(\widetilde G^{\oslash k}) \gtrsim \sqrt{k} \log  m \approx \sqrt{\log m \log N},
$$
where $N = |V(\widetilde G^{\oslash k})| = 2^{\Theta(k \log m)}$.
\end{corollary}

We remark that infinite families of $O(1)$-regular vertex-transitive graphs with $\mu_2 \geq \Omega(1)$
are well-known.  In particular, one can take any construction coming from the Cayley graphs of finitely
generated groups.  We refer to the survey \cite{HLW06}; see, in particular, Margulis' construction in Section 8.

\subsection{Extension to other $L_p$ spaces}
\label{sec:Lp}

Our previous lower bound dealt only with $L_2$.  We now prove the following.

\begin{theorem}
\label{thm:Lp}
If $G=(V,E)$ is an $O(1)$-regular $m$-vertex, vertex-transitive graph with $\mu_2 = \Omega(1)$,
for any $p > 1$, there exists a constant $C(p)$ such that
$$
c_p(\widetilde G^{\oslash k}) \gtrsim C(p) k^{1/q} \log  m \approx C(p) (\log m)^{1-1/q} (\log N)^{1/q}
$$
were $N = |V(\widetilde G^{\oslash k})|$ and $q = \max\{p,2\}$.
\end{theorem}

The only changes required are to Lemma \ref{lem:poincare} and Lemma \ref{lem:stretch} (which uses orthogonality).
The first can be replaced by Matou{\v{s}}ek's \cite{MatExp} Poincar\'e inequality:  If $G=(V,E)$ is an $O(1)$-regular expander
graph with $\mu_2=\Omega(1)$, then for any $p \in [1,\infty)$
and $f : V \to L_p$,
$$
\E_{x,y \in V}\, \|f(x)-f(y)\|_p^p \leq O(2p)^p\, \E_{(x,y) \in E} \,\|f(x)-f(y)\|_p^p.
$$
Generalizing Lemma \ref{lem:stretch} is more involved.
We need the following well-known 4-point inequalities for $L_p$ spaces.

\begin{lemma}\label{lem:Lp}
Consider any $p \geq 1$ and $u,v,w,x \in L_p$.  If $1 \leq p \leq 2$, then
\begin{equation}
\label{eq:1p2}
\|u-w\|_p^2 + (p-1) \|x-v\|_p^2 \leq \|u-v\|_p^2 + \|v-w\|_p^2 + \|x-w\|_p^2 + \|u-x\|_p^2.
\end{equation}
If $p \geq 2$, then
\begin{equation}
\label{eq:pg2}
\|u-w\|_p^p + \|x-v\|_p^p \leq 2^{p-2} \left(\|u-v\|_p^p + \|v-w\|_p^p + \|x-w\|_p^p + \|u-x\|_p^p\right).
\end{equation}
\end{lemma}

\begin{proof}
The following inequalities are known for $a,b \in L_p$ (see, e.g. \cite{BCL94}).  If $1 \leq p \leq 2$, then
$$\left\|\frac{a+b}{2}\right\|_p^2 + (p-1) \left\|\frac{a-b}{2}\right\|_p^2 \leq \frac{\left\|a\right\|_p^2+\left\|b\right\|_p^2}{2}.$$
On the other hand, if $p \geq 2$, then
$$\left\|\frac{a+b}{2}\right\|_p^p + \left\|\frac{a-b}{2}\right\|_p^p \leq \frac{\|a\|_p^p + \|b\|_p^p}{2}.$$
In both cases, the desired 4-point inequalities are obtained by averaging
two incarnations of one of the above inequalities with
$a=u-v, b=v-w$ and then $a=u-x, b=x-w$ and using convexity of the $L_p$ norm (see, e.g. \cite[Lem. 2.1]{LNdiamond}).
\end{proof}

\begin{lemma}\label{lem:stretchLp}
Let $G$ be a vertex transitive graph, and suppose $p > 1$.
If $q = \max\{p,2\}$, then there exists a constant $K(p) > 0$
such that the following holds.
Let $f : V(\widetilde G) \to L_p$ be an injective mapping
and define
$\bar f : V(\widetilde G) \to L_p$ by
$$
\bar f(x) = \frac{1}{|\Aut(G)|^{1/p}}\left(f(\widetilde \pi x)\vphantom{\bigoplus}\right)_{\pi \in \Aut(G)}.
$$
Suppose that $\beta$ is such that for every $i \in [D+1],$ there exists a vertical edge $(u^{(i)}, v^{(i)})$
which satisfies $\|\bar f(u^{(i)})-\bar f(v^{(i)})\|_p \geq \beta$.
Then there exists a horizontal edge $(x,y) \in E(\widetilde G)$ such that
\begin{equation}\label{eq:stretchLp}
\frac{\|\bar f(x)-\bar f(y)\|_p^{q}}{d_{\widetilde G}(x,y)^{q}} \geq \frac{\|f(s)-f(t)\|_p^{q}}{d_{\widetilde G}(s,t)^q} + K(p) \beta^{q}.
\end{equation}
\end{lemma}

\begin{proof}
Let $D = \diam(G)$.
For simplicity, we assume that $D$ is even in what follows.
\begin{enumerate}
\item[(F1)] $\|\bar f(s)-\bar f(t)\|_p = \|f(s)-f(t)\|_p$
\item[(F2)] For all $u,v \in V$,
\begin{eqnarray*}
\|\bar f(s)-\bar f(v^{(1)})\|_p&=&\|\bar f(s)-\bar f(u^{(1)})\|_p,\\
\|\bar f(t)-\bar f(v^{(D+1)})\|_p&=&\|\bar f(t)-\bar f(u^{(D+1)})\|_p.
\end{eqnarray*}
\item[(F3)] For every $u,v \in V$, $i\in [D]$,
$$\|\bar f(v^{(i)})-\bar f(v^{(i+1)})\|_p=\|\bar f(u^{(i)})-\bar f(u^{(i+1)})\|_p.$$
\end{enumerate}

Fix some $r \in V$ and
let $\rho_0 = \|\bar f(s)-\bar f(r^{(1)})\|_p$,
$\rho_i =  \|\bar f(r^{(2i-1)})-\bar f(r^{(2i+1)})\|_p$
for $i=1,\ldots,D/2$,
$\rho_{D/2+1} = \|\bar f(t)-\bar f(r^{(D+1)})\|_p$.
Also let $\rho_{i,1} = \|\bar f(r^{(2i-1)})-\bar f(r^{(2i)})\|_p$ and
$\rho_{i,2} = \|\bar f(r^{(2i)})-\bar f(r^{(2i+1)})\|_p$ for
$i=1,\ldots,D/2$.

Note that, by (F2) and (F3) above, the values $\{\rho_i\}$ do not
depend on the representative $r \in V$.
In this case, we have
\begin{equation}\label{eq:triLp}
\sum_{i=0}^{D/2+1} \rho_i \geq \|\bar f(s)-\bar f(t)\|_p = 9\gamma D,
\end{equation}
where we put $\gamma = \frac{\|f(s)-f(t)\|_p}{d_{\widetilde G}(s,t)}$.  Note that $\gamma > 0$
since $f$ is injective.

Let $\delta = \delta(p)$ be a constant to be chosen shortly.
Recalling that $d_{\widetilde G}(s,t) = 9D$ and $d_{\widetilde G}(s,r^{(1)})=4D$, observe that
if $\rho_0^q \geq \left(1+\delta\frac{\beta^q}{\gamma^q}\right) (4 \gamma D)^q$, then
$$
\max \left(\frac{\|\bar f(s)-\bar f(s'))\|_p^q}{d_{\widetilde G}(s,s')^q} ,\frac{\|\bar f(s')-\bar f(r^{(1)})\|_p^q}{d_{\widetilde G}(s',r^{(1)})^q}\right) \geq \gamma^q + \delta \beta^q,
$$
verifying \eqref{eq:stretchLp}.  The symmetric argument holds for $\rho_{D/2+1}$, thus
we may assume that $$\rho_0, \rho_{D/2+1} \leq 4 \gamma D\left(1+\delta\frac{\beta^q}{\gamma^q}\right)^{1/q}
\leq 4 \gamma D \left(1+\delta\frac{\beta^q}{\gamma^q}\right).
$$
Similarly, we may assume that $\rho_{i,1},\rho_{i,2} \leq \gamma \left(1+\delta\frac{\beta^q}{\gamma^q}\right)^{1/q}$ for every $i \in [D/2]$.

In this case, by \eqref{eq:triLp}, there must exist an index $j \in \{1,2,\ldots,D/2\}$ such that
$$\rho_{j} \geq \left(1-8\delta\frac{\beta^q}{\gamma^q}\right) 2\gamma.$$

Now, consider a vertical edge $(u^{(2j)}, v^{(2j)})$ with
$\|f(u^{(2j)})-f(v^{(2j)})\|_p \geq \beta$.  Also consider the
vertices $v^{(2j-1)}$ and $v^{(2j+1)}$.  We now replace the use
of orthogonality ((F4) in Lemma \ref{lem:stretch}) with Lemma \ref{lem:Lp}.

We apply one of \eqref{eq:1p2} or \eqref{eq:pg2} of these two inequalities with $x=f(u^{(2j)}),v=f(v^{(2j)}),u=f(v^{(2j-1)}),w=f(v^{(2j+1)})$.
In the case $p \geq 2$, we use \eqref{eq:1p2} to conclude that
\begin{eqnarray*}
\|f(u^{(2j)})\!-\!f(v^{(2j-1)})\|_p^p \!+\!
\|f(u^{(2j)})\!-\!f(v^{(2j+1)})\|_p^p\! &\geq&\! 2^{-p+2} \rho_j^p \!+\! 2^{-q+2} \beta^p
\!-\! \rho_{j,1}^p\! -\! \rho_{j,2}^p \\
&\geq &\! 2 \gamma^p + 2^{-p+2} \beta^p - 34 \delta p \beta^p.
\end{eqnarray*}
Thus choosing $\delta = \frac{2^{1-p}}{34p}$ yields the desired result
for one of $(u^{(2j)}, v^{(2j-1)})$ or $(u^{(2j)}, v^{(2j+1)})$.

\medskip

In the case $1 \leq p \leq 2$,
we use \eqref{eq:pg2} to conclude that
\begin{eqnarray*}
\|f(u^{(2j)})-f(v^{(2j-1)})\|_p^2 \!+\!
\|f(u^{(2j)})-f(v^{(2j+1)})\|_p^2 \geq \rho_j^2 + (p-1) \beta^2
- \rho_{j,1}^2 - \rho_{j,2}^2.
\end{eqnarray*}
A similar choice of $\delta$ again yields the desired result.
\end{proof}

\bibliographystyle{abbrv}
\bibliography{trees,gl}

\def\cprime{$'$}
\begin{thebibliography}{10}

\bibitem{ALN05}
S.~Arora, J.~R. Lee, and A.~Naor.
\newblock Euclidean distortion and the {S}parsest {C}ut.
\newblock {\em J. Amer. Math. Soc.}, 21(1):1--21, 2008.

\bibitem{AR98}
Y.~Aumann and Y.~Rabani.
\newblock An {$O(\log k)$} approximate min-cut max-flow theorem and
  approximation algorithm.
\newblock {\em SIAM J. Comput.}, 27(1):291--301 (electronic), 1998.

\bibitem{BCL94}
K.~Ball, E.~A. Carlen, and E.~H. Lieb.
\newblock Sharp uniform convexity and smoothness inequalities for trace norms.
\newblock {\em Invent. Math.}, 115(3):463--482, 1994.

\bibitem{CGR05}
S.~Chawla, A.~Gupta, and H.~R\"acke.
\newblock An improved approximation to sparsest cut.
\newblock In {\em Proceedings of the 16th Annual ACM-SIAM Symposium on Discrete
  Algorithms}, Vancouver, 2005. ACM.

\bibitem{GKL03}
A.~Gupta, R.~Krauthgamer, and J.~R. Lee.
\newblock Bounded geometries, fractals, and low-distortion embeddings.
\newblock In {\em 44th Symposium on Foundations of Computer Science}, pages
  534--543, 2003.

\bibitem{HLW06}
S.~Hoory, N.~Linial, and A.~Wigderson.
\newblock Expander graphs and their applications.
\newblock {\em Bull. Amer. Math. Soc. (N.S.)}, 43(4):439--561 (electronic),
  2006.

\bibitem{KLMN05}
R.~Krauthgamer, J.~R. Lee, M.~Mendel, and A.~Naor.
\newblock Measured descent: {A} new embedding method for finite metrics.
\newblock {\em Geom. Funct. Anal.}, 15(4):839--858, 2005.

\bibitem{Laakso}
T.~J. Laakso.
\newblock Plane with {$A\sb \infty$}-weighted metric not bi-{L}ipschitz
  embeddable to {${\Bbb R}\sp N$}.
\newblock {\em Bull. London Math. Soc.}, 34(6):667--676, 2002.

\bibitem{Lang}
U.~Lang and C.~Plaut.
\newblock Bilipschitz embeddings of metric spaces into space forms.
\newblock {\em Geom. Dedicata}, 87(1-3):285--307, 2001.

\bibitem{Lee05}
J.~R. Lee.
\newblock On distance scales, embeddings, and efficient relaxations of the cut
  cone.
\newblock In {\em SODA '05: Proceedings of the sixteenth annual ACM-SIAM
  symposium on Discrete algorithms}, pages 92--101, Philadelphia, PA, USA,
  2005. Society for Industrial and Applied Mathematics.

\bibitem{LeeVol}
J.~R. Lee.
\newblock Volume distortion for subsets of {E}uclidean spaces.
\newblock {\em Discrete Comput. Geom.}, 41(4):590--615, 2009.

\bibitem{LNdiamond}
J.~R. Lee and A.~Naor.
\newblock Embedding the diamond graph in ${L}_p$ and dimension reduction in
  ${L}_1$.
\newblock {\em Geom. Funct. Anal.}, 14(4):745--747, 2004.

\bibitem{LLR95}
N.~Linial, E.~London, and Y.~Rabinovich.
\newblock The geometry of graphs and some of its algorithmic applications.
\newblock {\em Combinatorica}, 15(2):215--245, 1995.

\bibitem{MatExp}
J.~Matou{\v{s}}ek.
\newblock On embedding expanders into {$l\sb p$} spaces.
\newblock {\em Israel J. Math.}, 102:189--197, 1997.

\bibitem{Mat01}
J.~Matou{\v{s}}ek.
\newblock {\em Lectures on discrete geometry}, volume 212 of {\em Graduate
  Texts in Mathematics}.
\newblock Springer-Verlag, New York, 2002.

\bibitem{NR03}
I.~Newman and Y.~Rabinovich.
\newblock A lower bound on the distortion of embedding planar metrics into
  {E}uclidean space.
\newblock {\em Discrete Comput. Geom.}, 29(1):77--81, 2003.

\bibitem{Rao99}
S.~Rao.
\newblock Small distortion and volume preserving embeddings for planar and
  {E}uclidean metrics.
\newblock In {\em Proceedings of the 15th Annual Symposium on Computational
  Geometry}, pages 300--306, New York, 1999. ACM.

\end{thebibliography}

\end{document}